\documentclass[11pt]{amsart}
\usepackage{amsmath,amssymb,amsfonts,graphics}

\newtheorem{thm}{Theorem}[section]
\newtheorem{cor}[thm]{Corollary}
\newtheorem{lem}[thm]{Lemma}
\newtheorem{prop}[thm]{Proposition}
\theoremstyle{definition}
\newtheorem{defn}[thm]{Definition}
\theoremstyle{remark}
\newtheorem{exm}[thm]{Example}
\newtheorem{rem}[thm]{Remark}
\numberwithin{equation}{section}

\newcommand{\norm}[1]{\left\Vert#1\right\Vert}

\newcommand{\R}{\mathbb{R}}
\newcommand{\N}{\mathbb{N}}
\newcommand{\Z}{\mathbb{Z}}
\newcommand{\T}{\mathbb{T}}
\newcommand{\C}{\mathbb{C}}
\newcommand{\Cm}{\mathbb{C}^*}
\newcommand{\Rm}{\mathbb{R}_+}

\newcommand{\Zc}{\mathcal{Z}}
\newcommand{\Nc}{\mathcal{N}}
\newcommand{\tw}{\circledast}
\newcommand{\tc}{\circledast_\T}

\newcommand{\Sm}{\mathcal{S}}
\newcommand{\Zb}{\Zc_{b}^2(G,\Cm)}

\newcommand{\Zbw}{\Zc_{bw}^2(G,\Cm)}
\newcommand{\ZTb}{\Zc_{b}^2(G,\mathbb{T})}

\newcommand{\Hc}{\mathcal{H}}

\newcommand{\om}{\omega}
\newcommand{\Om}{\Omega}
\newcommand{\sg} {\sigma}

\newcommand{\la}{\langle}
\newcommand{\ra}{\rangle}

\begin{document}

\title[Twisted Orlicz algebras and complete isomorphism to operator algebras]
{Twisted Orlicz algebras and complete isomorphism to operator algebras}

\author[Serap \"{O}ztop]{Serap \"{O}ztop}
\address{Department of Mathematics, Faculty of Science, Istanbul University, Istanbul, Turkey}
\email{oztops@istanbul.edu.tr}

\author{Ebrahim Samei}
\address{Department of Mathematics and Statistics, University of Saskatchewan, Saskatoon, Saskatchewan, S7N 5E6, Canada}
\email{samei@math.usask.ca}

\author{Varvara Shepelska}
\address{Department of Mathematics and Statistics, University of Saskatchewan, Saskatoon, Saskatchewan, S7N 5E6, Canada}
\email{shepelska@gmail.com}

\footnote{{\it Date}: \today.



The second named author was partially supported by NSERC Grant no. 409364-2015 and 2221-Fellowship Program For Visiting Scientists And Scientists On Sabbatical Leave from Tubitak, Turkey. The third named author was also partially supported by a PIMS Postdoctoral Fellowship at the University of Saskatchewan.}











\maketitle

\begin{abstract}
Let G be a locally compact group, let $\Om:G\times G\to \Cm$ be a 2-cocycle, and let ($\Phi$,$\Psi$) be a complementary pair of strictly increasing continuous Young functions. It is shown in \cite{OS2} that $(L^\Phi(G),\tw)$ becomes an Arens regular dual Banach algebra if
\begin{align}\label{Eq:2-cocycle bdd sum-abstract}
|\Om(s,t)|\leq u(s)+v(t) \ \ \ (s,t\in G)
\end{align}
for some $u,v\in \Sm^\Psi(G)$. We prove if $L^\Phi(G)\subseteq L^2(G)$ and $u,v$ in \eqref{Eq:2-cocycle bdd sum-abstract} can be chosen to belong to $L^2(G)$, then $(L^\Phi(G),\tw)$ with the maximal operator space structure is completely isomorphic to an operator algebra. We also present further classes of 2-cocycles for which one could obtain such algebras generalizing in part the results of \cite{OS1}. We apply our methods to compactly generated group of polynomial growth and demonstrate that our results could be applied to variety of cases.
\end{abstract}

Orlicz spaces represent an important class of Banach function spaces considered in mathematical analysis. This class naturally arises as a generalization of $L^p$-spaces and contains, for example, the well-known Zygmund space $L \log^+ L$ which is a Banach space related to Hardy-Littlewood maximal functions. Orlicz spaces can also contain certain Sobolev spaces as subspaces. Linear properties of Orlicz spaces have been studied thoroughly (see \cite{rao} for example). However, until recently, little attention has been paid to their possible algebraic properties, particularly, if they are considered over translation-invariant measurable spaces. One reason might be that, on its own, an Orlicz space is rarely an algebra with respect to a natural product! For instance, it is well-known that for a locally compact group $G$, $L^p(G)$ ($1<p<\infty$) is an algebra under the convolution product exactly when $G$ is compact \cite{S}. Similar results have also been obtained for other classes of Orlicz algebras (see \cite{AM}, \cite{HKM}, \cite{S} for details).

The preceding results indicate that, in most cases, Orlicz spaces over locally compact groups are simply ``too big" to become algebras under convolution. However, it turned out that it is possible for ``weighted" Orlicz spaces to become algebras. In fact, weighted $L^p$-algebras and their properties have been studied by many authors including J. Wermer on the real line and
Yu. N. Kuznetsova on general locally compact groups (see, for example, \cite{K1}, \cite{K2}, \cite{KM}, \cite{W} and the references therein). These spaces have various properties and numerous applications in harmonic analysis. For instance, by applying the Fourier transform, we know that Sobolev spaces $W^{k,2}(\T)$ are nothing but certain weighted $l_\om^2(\Z)$ spaces.

Recently, in \cite{OO}, A. Osan\c{c}l{\i}ol and S. \"{O}ztop considered weighted Orlicz algebras over locally compact groups and studied their properties, extending, in part, the results of \cite{K1} and \cite{K2}. In \cite{OS1} and \cite{OS2}, the first two-named authors initiated a more general approach by considering the twisted convolution coming from a 2-cocycle $\Om$ with values in $\C^*$, the multiplicative group of complex numbers. It is shown in \cite{OS1} if ($\Phi$,$\Psi$) is a complementary pair of strictly increasing continuous Young functions and
\begin{align}\label{Eq:2-cocycle bdd sum-Introduction}
|\Om(s,t)|\leq u(s)+v(t) \ \ \ (s,t\in G).
\end{align}
for some $u,v\in L^\Psi(G)$, then $L^\Phi(G)$, together with the twisted convolution $\tw$ coming from $\Om$, becomes a Banach algebra or a Banach $*$-algebra
\cite[Theorems 3.3 and 4.5]{OS1}; we called them {\it twisted Orlicz algebras}.
These methods produce abundant families of symmetric Banach $*$-algebras in the form of twisted Orlicz algebras, mostly on compactly generated groups with polynomial growth \cite[Theorems 5.2 and 5.8]{OS1}. Moreover, if we can choose $u,v$ in \eqref{Eq:2-cocycle bdd sum-Introduction} in the subspace $\Sm^\Psi(G)$ of $L^\Psi(G)$ (see Section \ref{S:Prelim-Orlicz Space} for definition), then $(L^\Phi(G),\tw)$ becomes an Arens regular dual Banach algebra \cite[Theorem 4.2 and 5.3]{OS2}.

In this paper, we present a general method to obtain 2-cocycles on compactly generated groups with polynomial growth satisfying \eqref{Eq:2-cocycle bdd sum-Introduction}. More precisely, if $G$ is a compactly generated group of polynomial growth and $\tau$ is the length function on $G$ (see \eqref{Eq:length function}), then we consider those 2-cocycles $\Om$ for which $|\Om|$ is a 2-coboundary determined by weights of the from
$$\om(s)=e^{\rho(\tau(s))} \ \ \ \ \ (s\in G),$$
where $\rho:[0,\infty)\to [0,\infty)$ is an increasing concave function with $\rho(0)=0$. In this case, we show that for
$$u(s)=v(s)=e^{[\rho(2\tau(s))-2\rho(\tau(s))]} \ \ \ (s\in G),$$
$\Om$ satisfies the inequality \eqref{Eq:2-cocycle bdd sum-Introduction}. Moreover, we find criterions which enforce $u\in \Sm^\Psi(G)$ (see Theorems \ref{T:subexpo weight decreasing quotient1} and \ref{T:subexpo weight decreasing quotient}). This provided many families of Arens regular dual twisted Orlicz algebra extending, in part, the results of \cite{OS1} and \cite{OS2}. Furthermore, in Section \ref{S:Operator algebra}, we show that if $L^\Phi(G)\subseteq L^2(G)$ and $u,v$ in \eqref{Eq:2-cocycle bdd sum-Introduction} can be chosen to belong to $L^2(G)$, then $(L^\Phi(G),\tw)$ is satisfies a much stronger property namely, it becomes completely isomorphic to an operator algebra. Here the operator space structure on Orlicz spaces are the maximal one. We also present a wide range of examples of such twisted Orlicz algebras.
In particular, we obtain certain weighted $L^p$-spaces that are completely isomorphic to operators algebras.

\section{Preliminaries}\label{S:Prelim}

In this section, we give some definitions and state some technical results that will be crucial in the rest of this paper.


\subsection{Orlicz Spaces}\label{S:Prelim-Orlicz Space}

We start by recalling some facts concerning Young functions and Orlicz spaces. Our main reference is \cite{rao}.

A nonzero function $\Phi:[0,\infty) \to[0,\infty]$ is called a Young
function if $\Phi$ is convex, $\Phi(0)=0$, and $\lim_{x\to \infty} \Phi(x)=\infty$. For a Young function $\Phi$, the complementary function $\Psi$ of
$\Phi$ is given by
\begin{align}\label{Eq:Young function-complementary}
\Psi(y)=\sup\{xy-\Phi(x):x\ge0\}\quad(y\geq 0).
\end{align}
It is easy to check that $\Psi$ is again a Young function. Also, if $\Psi$ is the complementary function of $\Phi$, then $\Phi$ is
the complementary of $\Psi$ and $(\Phi,\Psi)$ is called a
complementary pair, and for such functions we have the following Young inequality:
\begin{align}\label{Eq:Young inequality}
xy\le\Phi(x)+\Psi(y)\quad(x,y\ge0).
\end{align}
A Young function can have value $\infty$ at a point,
and hence be discontinuous at such a point. However, we always consider the pair of complementary Young
functions $(\Phi,\Psi)$ with both $\Phi$ and $\Psi$ being continuous and strictly increasing. In particular, they attain positive values on $(0,\infty)$.

Now suppose that $G$ is a locally compact group with a fixed Haar measure $ds$ and $(\Phi,\Psi)$ is a complementary pair of Young functions. We define
\begin{align}\label{Eq:Orlicz defn-0}
\mathcal{L}^\Phi(G)=\left\{f:G\to\C:f \ \text{is measurable and}\  \int_G\Phi(|f(s)|)\,ds <\infty
\right\}.
\end{align}
Since $\mathcal{L}^\Phi(G)$ is not always a linear space, we define the Orlicz space $L^\Phi(G)$ to be
\begin{align}\label{Eq:Orlicz defn}
L^\Phi(G)=\left\{f:G\to\C:\int_G\Phi(\alpha|f(s)|)\,ds <\infty
\mbox{ for some }\alpha>0\right\},
\end{align}
where $f$ indicates a member in an equivalence class of measurable functions with respect to the Haar measure $ds$. The Orlicz space becomes a Banach space under the (Orlicz) norm $\|\cdot\|_\Phi$
defined for $f\in L^\Phi(G)$ by
\begin{align}\label{Eq:Orlicz norm}
\|f\|_\Phi=\sup\left\{\int_G|f(s)v(s)|\,ds: \int_G\Psi(|v(s)|)\,ds \le1\right\},
\end{align}
where $\Psi$ is the complementary function to $\Phi$. One can also
define the (Luxemburg) norm $N_\Phi(\cdot)$ on $L^\Phi(G)$ by
\begin{align}\label{Eq:Orlicz Luxemburg defn}
N_\Phi(f)=\inf\left\{k>0:\int_G\Phi\left(\frac{|f(s)|}{k}\right)
\,ds \le1\right\}.
\end{align}
It is known that these two norms are equivalent, that is,
\begin{align}\label{Eq:Orlicz norm-Luxemburg relation}
N_\Phi(\cdot)\le \|\cdot\|_\Phi\le2 N_\Phi(\cdot)
\end{align}
and
\begin{align}\label{Eq:Orlicz norm-defn relation}
N_\Phi(f)\le1 \ \ \text{if and only if}\ \ \int_G\Phi(|f(s)|)\,ds \le1.
\end{align}
Let $\Sm^\Phi(G)$ be the closure of the linear
space of all step functions in $L^\Phi(G)$. Then $\Sm^\Phi(G)$ is a Banach space
and contains $C_c(G)$, the space of all continuous functions on $G$ with compact support, as a dense subspace \cite[Proposition 3.4.3]{rao}. Moreover, $\Sm^\Phi(G)^*$, the dual of  $\Sm^\Phi(G)$, can be identified with $L^\Psi(G)$ in a natural way \cite[Theorem 4.1.6]{rao}. Another useful characterization of $\Sm^\Phi(G)$ is that
$f\in \Sm^\Phi(G)$ if and only if for every $\alpha>0$, $\alpha f\in \mathcal{L}^\Phi(G)$ \cite[Definition 3.4.2 and Proposition 3.4.3]{rao}.


%

In general, there is a straightforward method to construct various complementary pairs of strictly increasing continuous Young functions as described in
\cite[Theorem 1.3.3]{rao}. Suppose that $\varphi: [0,\infty)\to [0,\infty)$ is a continuous strictly increasing function with $\varphi(0)=0$ and
$\lim_{x\to \infty} \varphi(x)=\infty.$
Then $$\Phi(x)=\int_0^x \varphi(y)dy$$ is a continuous strictly increasing Young function and
$$\Psi(y)=\int_0^y \varphi^{-1}(x)dx$$
is the complementary Young function of $\Phi$ which is also continuous and strictly increasing.
Here $\varphi^{-1}(x)$ is the inverse function of $\varphi$. Here are several families of examples satisfying the above construction (see \cite[Proposition 2.11]{ML} and \cite[Page 15]{rao} for more details):

$(1)$ For $1< p,q<\infty$ with $\frac{1}{p}+\frac{1}{q}=1$, if $\Phi(x)=\frac{x^p}{p}$, then $\Psi(y)=\frac{y^q}{q}$. In this case,
the space $L^{\Phi}(G)$ becomes the Lebesgue space $L^p(G)$ and the norm $\|\cdot\|_{\Phi}$ is equivalent to the
classical norm $\|\cdot\|_{p}$.

(2) If $\Phi(x)=x\ln (1+x)$, then $\Psi(x) \approx \cosh x-1$.


(3) If $\Phi(x)=e^x-x-1$, then $\Psi(x)=(1+x)\ln(1+x)-x$.


\subsection{2-Cocycles and 2-Cobounaries}

Throughout this paper, we use the following notation: $\Cm$ denotes the multiplicative group of complex numbers, i.e. $\Cm=\C \setminus \{0\}$,
$\Rm$ stands for the multiplicative group of positive real numbers, and $\T$ denotes the unit circle in $\C$.

\begin{defn}\label{D:2-cocycle defn}
Let $G$ and $H$ be locally compact groups such that $H$ is abelian. A {\it (normalized) 2-cocycle on $G$ with values in $H$} is a Borel
measurable map $\Om: G\times G \to H$ such that
\begin{align}\label{Eq:2-cocycle relation}
\Om(r,s)\Om(rs,t)=\Om(s,t)\Om(r,st) \ \ \  (r,s,t \in G)
\end{align}
and
\begin{align}\label{Eq:2-cocycle relation normalization}
\Om(r,e_G)=\Om(e_G,r)=e_H \ \ \ (r\in G).
\end{align}
The set of all normalized 2-cocycles will be denoted by $\Zc^2(G,H)$.
\end{defn}
If $\om: G \to H$ is measurable with $\om(e_G)=e_H$, then it is easy to see that the mapping
$$(s,t)\mapsto \om(st)\om(s)^{-1}\om(t)^{-1}$$
satisfies \eqref{Eq:2-cocycle relation} and \eqref{Eq:2-cocycle relation normalization}. Hence, it is a
2-cocycle; such maps are called \mbox{\it 2-coboundaries}. The set of 2-coboundaries will be denoted by $\Nc^2(G,H)$.
It is easy to check that $\Zc^2(G,H)$ is an abelian group under the product
$$\Om_1 \Om_2 (s,t)=\Om_1(s,t)\Om_2(s,t) \ \ (s,t\in G),$$
and $\Nc^2(G,H)$ is a (normal) subgroup of $\Zc^2(G,H)$. This, in particular, implies that
$$\Hc^2(G,H):=\Zc^2(G,H)/\Nc^2(G,H)$$ turns into a group. The latter is called the {\it 2$^{nd}$ group cohomology of $G$ into~$H$}
with the trivial actions (i.e. $s\cdot \alpha=\alpha\cdot s=\alpha$ for all $s\in G$ and $\alpha \in H$).

We are mainly interested in the cases when $H$ is $\Cm$, $\Rm$, or $\T$. One essential observation is that
we can view $\Cm=\Rm \T$ as a (pointwise) direct product of groups. Hence, for any 2-cocycle
$\Om$ on $G$ with values in $\Cm$ and $s,t\in G$, we can (uniquely) write
$\Om(s,t)=|\Om(s,t)| e^{i\theta}$ for some $0\leq \theta <2\pi$. Therefore, if we put
\begin{align}
|\Om|(s,t):=|\Om(s,t)| \ \ \text{and}\ \ \Om_\T(s,t):=e^{i\theta},
\end{align}
then $\Om=|\Om|\Om_\T$ (in a unique way) and the mappings $|\Om|$ and $\Om_\T$ are \mbox{2-cocycles}
on $G$ with values in $\Rm$ and $\T$ respectively.





\subsection{Twisted Orlicz algebras}\label{S:Twisted Orlicz alg}

In this section we present and summarize what we need from the theory of twisted Orlicz algebras. These are taken from \cite{OS1}. Throughout this section, $G$ is a locally compact group with a fixed left Haar measure $ds$.

\begin{defn}\label{D:bound 2 cocycles}
We denote $\Zb$ to be the group of {\it bounded 2-cocycles on $G$ with values in $\Cm$}
which consists of all element $\Om\in \Zc^2(G,\Cm)$ satisfying the following conditions:\\
$(i)$ $\Om \in L^\infty(G\times G)$;\\
$(ii)$ $\Om_\T$ is continuous.\\
We also define $\Zbw$ to be the subgroup of $\Zb$ consisting of elements $\Om \in \Zb$ for which
$$|\Om|(s,t)=\frac{\om(st)}{\om(s)\om(t)} \ \ \ (s,t\in G),$$ where $\om : G \to \R_+$ is a locally integrable measurable function
with $\om(e)=1$ and $1/\om\in L^\infty(G)$. In this case, we call $\om$ a {\it weight} on $G$ and say that $|\Om|$ is the {\it 2-coboundary determined by} $\om$, or alternatively, $\om$ is the {\it weight associated to} $|\Om|$. We also say that $\om$ is {\it symmetric} if $\om(s)=\om(s^{-1})$ for all $s\in G$.
\end{defn}

Now suppose that $\Om\in \Zb$ and $f$ and $g$ are measurable functions on $G$.
We define the {\it twisted convolution of $f$ and $g$ under $\Om$} to be
\begin{align}\label{Eq:twisted convolution}
f\tw g (t)=\int_G f(s)g(s^{-1}t)\Om(s,s^{-1}t)ds  \ \ \ (t\in G)
\end{align}
It follows routinely that for every $f,g\in L^1(G)$,
$f\tw g\in L^1(G)$ with $\|f\tw g\|_1\leq \|\Om\|_\infty \|f\|_1\|g\|_1$.
We conclude that $(L^1(G),\tw)$ becomes a Banach algebra; it is called
the {\it twisted group algebra} (see \cite[Section 2]{OS1} for more details).

\begin{defn}
Let $\Om\in \Zb$ and $\tw$ be the twisted convolution coming from $\Om$. We say
that $(L^\Phi(G),\tw)$ is a {\it twisted Orlicz algebra} if $(L^\Phi(G),\tw, \|\cdot\|_\Phi)$ is a Banach algebra, i.e. there is $C>0$ such that
for every $f,g\in L^\Phi(G)$, $f\tw g\in L^\Phi(G)$ with
$$ \|f\tw g\|_\Phi \leq C\|f\|_\Phi \|g\|_\Phi.$$
\end{defn}

In \cite[Lemma 3.2 and Theorem 3.3]{OS1}, sufficient conditions on $\Om$ were found under which the twisted convolution \eqref{Eq:twisted convolution} turns an Orlicz space to a twisted Orlicz algebra.

\begin{thm}\label{T:twisted Orlicz alg}
Let $G$ be a locally compact group and $\Om\in \Zb$.\\
$(i)$ $L^\Phi(G)$ is a Banach $L^1(G)$-bimodule with respect to the twisted convolution \eqref{Eq:twisted convolution} having $\Sm^\Phi(G)$ as an essential Banach $L^1(G)$-submodule.\\
$(ii)$ Suppose that there exist non-negative measurable functions $u$ and $v$
in $L^\Psi(G)$ such that
\begin{align}\label{Eq:2-cocycle bdd sum}
|\Om(s,t)|\leq u(s)+v(t) \ \ \ (s,t\in G).
\end{align}
Then
for every $f,g\in L^\Phi(G)$, the twisted convolution \eqref{Eq:twisted convolution}
is well-defined on $L^\Phi(G)$
so that $(L^\Phi(G),\tw)$ becomes a twisted Orlicz algebra having $\Sm^\Phi(G)$ as a closed subalgebra.
\end{thm}


\section{Twisted Orlicz algebras on groups with polynomial growth}\label{S:Twisted Orlicz alg-PG}

Let $G$ be a compactly generated group with a fixed compact symmetric generating neighborhood $U$ of the identity.
$G$ is said to have {\it polynomial growth} if there exist $C>0$ and $d\in \N$ such that for every $n\in \N$
	$$\lambda(U^n)\leq Cn^d \ \ \ (n\in \N).$$
Here $\lambda(S)$ is the Haar measure of any measurable $S\subseteq G$ and
	$$U^n=\{u_1\cdots u_n : u_i\in U, i=1,\ldots, n \}.$$
The smallest such $d$ is called {\it the order of growth} of $G$ and is denoted by~$d(G)$.
It can be shown that the order of growth of $G$ does not depend on the symmetric generating set $U$, i.e. it is a universal constant for $G$.

It is immediate that compact groups are of polynomial growth. More generally, every $G$ with the property that the conjugacy class of
every element in $G$ is relatively compact has polynomial growth \cite[Theorem 12.5.17]{Pal}. Also every (compactly generated) nilpotent
group (hence an abelian group) has polynomial growth \cite[Theorem 12.5.17]{Pal}.

Using the generating set $U$ of $G$ we can define a {\it length function} $\tau_U : G \to [0, \infty)$ by
\begin{align}\label{Eq:length function}
\tau_U(s)=\inf \{n\in \N : s\in U^n \} \ \ \text{for} \ \ s \neq e, \ \ \tau_U(e)=0.
\end{align}
When there is no fear of ambiguity, we write $\tau$ instead of $\tau_U$.
It is straightforward to verify that $\tau$ is a symmetric subadditive function on $G$, i.e.
\begin{align}\label{Eq:lenght func-trai equality}
\tau(st)\leq \tau(s)+\tau(t) \ \ \text{and} \ \ \tau(s)=\tau(s^{-1})\ \ (s,t\in G).
\end{align}
The length function $\tau$ can be used to construct many classes of weights on~$G$. In fact, if $\rho:\N\cup \{0\}\to \R^+$ is an increasing subadditive function with $\rho(0)=0$, then
\begin{align}\label{Eq:weight-lenght function}
\om(s)=e^{\rho(\tau(s))} \ \ \ \ \ (s\in G)
\end{align}
is a weight on $G$. For instance, for every $0< \alpha < 1$, $\beta > 0$, $\gamma >0$, and $C>0$,
we can define the {\it polynomial weight} $\om_\beta$ on $G$ of order $\beta$ by
\begin{align}\label{Eq:poly weight-defn}
\om_\beta(s)=(1+\tau(s))^\beta  \ \ \ \ (s\in G),
\end{align}
and the {\it subexponential weights} $\sg_{\alpha, C}$ and $\nu_{\beta,C}$ on $G$ by
\begin{align}\label{Eq:Expo weight-defn}
\sg_{\alpha,C}(s)=e^{C\tau(x)^\alpha} \ \ \ \ (s\in G),
\end{align}
\begin{align}\label{Eq:Expo weight II-defn}
\nu_{\gamma,C}(s)=e^\frac{C\tau(s)}{(\ln (1+\tau(s)))^\gamma} \ \ \ \ (s\in G).
\end{align}

In \cite[Chapter 5]{OS1} it is shown how one could apply condition \eqref{Eq:2-cocycle bdd sum} and Theorem \ref{T:twisted Orlicz alg} to obtain twisted Orlicz algebras to compactly generated groups of polynomial growth. In this section we will extend the methods presented in \cite[Chapter 5]{OS1} and find a general criterion that can be applied to weights of the form \eqref{Eq:weight-lenght function}. First, we need the following lemma.

\begin{lem}\label{L:weight decreasing quotient}
Suppose that $\sg$ is a weight on $\N_0:=\N \cup \{0\}$ such that the sequence
\[
\left\{\frac{\sg(n+1)}{\sg(n)}\right\}_{n\in \N_0}
\]
is decreasing.
Then for all $m,n\in \N_0$,
\begin{align}\label{Eq:weight decreasing quotient-1}
\frac{\sg(m+n)}{\sg(m)\sg(n)}\leq \frac{\sg(2m)}{\sg(m)^2}+\frac{\sg(2n)}{\sg(n)^2}.
\end{align}
\end{lem}

\begin{proof}
Clearly, the result follows if either $m$ or $n$ is zero.
For simplicity, let
$$\mu_0=1 \ \text{and}\ \mu_n=\frac{\sg(n)}{\sg(n-1)} \ \ \ (n\in \N).$$
Suppose that $m\geq n\geq 1$. Then, by our assumption, we have that $\mu_{m+k}\leq \mu_{n+k}$ for all $0\le k\le n$. Hence,
$$ \prod\limits_{k=0}^{n}\mu_k \prod\limits_{k=1}^{n}\mu_{m+k} \leq \prod\limits_{k=0}^{2n} \mu_k.$$
Therefore,
$$ \prod\limits_{k=0}^{n}\mu_k \prod\limits_{k=0}^{m+n}\mu_{k} \leq  \prod\limits_{k=0}^{m}\mu_k \prod\limits_{k=0}^{2n} \mu_k,$$
or, equivalently,
$$\sg(n)\sg(m+n)\leq \sg(m)\sg(2n).$$
The last inequality implies that
$$\frac{\sg(m+n)}{\sg(m)\sg(n)}\leq \frac{\sg(2n)}{\sg(n)^2}.$$
Similarly, when  $n\geq m\geq 1$, we get
$$\frac{\sg(m+n)}{\sg(m)\sg(n)}\leq \frac{\sg(2m)}{\sg(m)^2}.$$
Putting the above two relations together, we get \eqref{Eq:weight decreasing quotient-1}.
\end{proof}

In order to obtain twisted Orlicz algebras, a critical part of our argument is the condition \eqref{Eq:2-cocycle bdd sum}. Even though this condition might be difficult to verify in general, we present a simple yet useful decomposition which can be applied to a wide class of weights.

\begin{thm}\label{T:subexpo weight decreasing quotient1}
Let $G$ be a compactly generated group of polynomial growth, $\rho:[0,\infty)
\to [0,\infty)$ be an increasing concave function with $\rho(0)=0$,
and $\om$ be the weight on $G$ defined in \eqref{Eq:weight-lenght function}.
Then, for every $s,t\in G$,
\begin{align}\label{Eq:subexpo weight decreasing quotient-1}
\frac{\om(st)}{\om(s)\om(t)}\leq u(s)+u(t),
\end{align}
where
\begin{align}\label{Eq:subexpo weight decreasing quotient-2}
u(s)=e^{[\rho(2\tau(s))-2\rho(\tau(s))]} \ \ \ (s\in G).
\end{align}
\end{thm}

\begin{proof}
Let $\sg:[0,\infty) \to [0,\infty)$ be the function defined by
$$\sg(x)=e^{\rho(x)} \ \ \ (x\geq 0).$$
Since $\rho$ is concave, for every $n\in \N_0$, we have
$$\frac{\rho(n+2)+\rho(n)}{2}\leq \rho(n+1).$$
Therefore, $\left\{\frac{\sg(n+1)}{\sg(n)}\right\}_{n\in \N_0}=\{e^{\rho(n+1)-\rho(n)}\}_{n\in \N_0}$ is decreasing and so, by Lemma~\ref{L:weight decreasing quotient}, $\sg_{|_{\N_0}}$ satisfies
\eqref{Eq:weight decreasing quotient-1}.
Now take $s,t\in G$ and put $m=\tau(s)$ and $n=\tau(t)$. Since $\rho$ is increasing
and $\tau(st)\leq \tau(s)+\tau(t)$, we have
$$\om(st)=e^{\rho(\tau(st))}\leq e^{\rho(m+n)}=\sg(m+n).$$
Therefore,
$$\frac{\om(st)}{\om(s)\om(t)}\leq \frac{\sg(m+n)}{\sg(m)\sg(n)}
\leq \frac{\sg(2m)}{\sg(m)^2}+\frac{\sg(2n)}{\sg(n)^2}=u(s)+u(t).$$
\end{proof}







We are now ready to present the main result of this section. 

\begin{thm}\label{T:subexpo weight decreasing quotient}
Let $G$, $\rho$, and $\om$ be as in Theorem \ref{T:subexpo weight decreasing quotient1} and $\Om_\T\in\ZTb$. Suppose that
$u\in \Sm^\Psi(G)$, where $u$ is the function defined in \eqref{Eq:subexpo weight decreasing quotient-2}. Then $(L^\Phi(G), \tw)$ is
an Arens regular dual Banach algebra. This, in particular, happens in either of the following cases:\\
$(i)$ $u\om\in L^\infty(G)$ and $\om^{-1}\in \Sm^\Psi(G)$;\\
$(ii)$ $\rho$ is differentiable on $\R^+$ and the function
\begin{align}\label{Eq:subexpo weight decreasing quotient-diff}
v(s)=e^{[\tau(s)^2 q'(\tau(s))]} \ \ \ (s\in G)
\end{align}
belongs to $\Sm^\Psi(G)$, where $q(x)=\rho(x)/x$ on $\R^+$;\\
$(iii)$ $\rho$ is twice-differentiable on $\R^+$ and
\begin{align}\label{Eq:subexpo weight decreasing quotient-second derv}
\lim_{x\to \infty} x^2\rho''(x)<-d/l,
\end{align}
\noindent where $d:=d(G)$ is the order of growth of $G$ and
$l\geq 1$ is such that $\lim_{x\to 0^+}\frac{\Psi(x)}{x^l}$ exists. Here we may allow $\lim_{x\to \infty} x^2\rho''(x)$ to be $-\infty$.
\end{thm}

\begin{proof}
It follows from our hypothesis, \eqref{Eq:subexpo weight decreasing quotient-1}, and \cite[Theorems 4.2 and 5.3]{OS2} that $(L_{\om}^\Phi(G), \tc)$ is an Arens regular dual Banach algebra.
Hence it remains to prove that any one of the conditions (i)-(iii) implies that $u\in \Sm^\Psi(G)$.

If (i) holds, then clearly
$u\in \Sm^\Psi(G)$ since $\Sm^\Psi(G)$ is a $L^\infty(G)$-module under pointwise product.

Now suppose that (ii) holds. Then, for every $x>0$,
\begin{eqnarray*}
\rho(2x)-2\rho(x) &=& 2x[q(2x)-q(x)] \\
&=& \frac{q(2x)-q(x)}{(-2x)^{-1}-(-x)^{-1}} \\
&=& y^2q'(y),
\end{eqnarray*}
for some $y\in (x,2x)$. However, it is easy to see that $(x^2q'(x))'=x\rho''(x)<0$ so that
$x^2q'(x)$ is decreasing on $\R^+$. Therefore,
$$\rho(2x)-2\rho(x) \leq x^2q'(x) \ \ \ (x>0).$$
This, together with \eqref{Eq:subexpo weight decreasing quotient-2} and \eqref{Eq:subexpo weight decreasing quotient-diff}, imply that $0<u\leq v$ and so $u\in \Sm^\Psi(G)$ by \cite[Definition 3.4.2 and Proposition 3.4.3]{rao}.

Finally, suppose that (iii) holds. Then $\lim_{x\to\infty} x^2q'(x)=-\infty$. Indeed, since $(x^2q'(x))'=x\rho''(x)$, (iii) implies that there exists $x_0>0$ and $a>0$ such that $(x^2q'(x))'\le -\frac a x$ for every $x\ge x_0$, and the desired result follows by integration. Hence we can apply the L'Hospital's Rule to obtain
$$ \lim_{x\to \infty} \frac{x^2q'(x)}{\ln (1+x)}=\lim_{x\to \infty} \frac{x\rho''(x)}{(1+x)^{-1}}=\lim_{x\to \infty} x^2\rho''(x) <-d/l.$$
Therefore, there is $k<-d/l$ such that when $x$ is large enough,
$$e^{x^2q'(x)}\leq (1+x)^{k} .$$
So, by part (ii), it suffices to prove that the function
$$f(s)=(1+\tau(s))^{k} \ \ \ (s\in G)$$
belongs to $\Sm^\Psi(G)$. However, this is verified in the proof of \cite[Corollary 5.3]{OS1}. This completes our proof.
\end{proof}

\begin{rem}
(i) Since twisted Orlicz spaces generalize weighted $L^p$-spaces where one natural condition related to $L^p(G,\omega)$ becoming a Banach algebra is that $\omega^{-1}\in L^q(G)$ (see \cite[Theorem~3]{K1}), it would have been nice to have that $(L_{\om}^\Phi(G), \tc)$ becomes a Banach algebra provided $\omega^{-1}\in\Sm^\Psi(G)$. However, an example presented in Appendix A shows that the methods of Theorem~\ref{T:subexpo weight decreasing quotient} cannot guarantee such a result even under some natural additional assumptions on the function $\rho$.

(ii) we note that Theorem~\ref{T:subexpo weight decreasing quotient} provides a different proof of \cite[Corollaries 4.3 and 5.4]{OS2}. Also it gives rise to new classes of twisted Orlicz algebras.
\end{rem}

\section{Operator algebra}\label{S:Operator algebra}

\subsection{Operator spaces} We will briefly remind the reader about the basic properties of operator spaces. We refer the reader to \cite{ER} for further details concerning the notions presented below.

An {\it (abstract) operator space} is a vector space $V$ together with a family $\{\norm{\cdot}_n\}$ of Banach space norms on $M_n(V)$ stratifying certain relations (see \cite[p. 20]{ER}). Let $V$ and $W$ be operator spaces, and let $\theta: V \to W$ be a linear map. The completely bounded norm of $\theta$ is defined by
\[
\norm{ \theta }_{cb}=\sup_n \{\norm{ \theta_{n}} \}
\]
where $\theta _{n}:M_{n}(V)\rightarrow M_{n}(W)$ is given by
\[
\theta_{n}([v_{ij}])=[\theta(v_{ij})].
\]
We say that $\theta$ is completely bounded if $\norm{\theta}_{cb}<\infty ;$
is completely contractive if $\norm{\theta}_{cb}\leq 1$ and is a complete isometry if each $\theta_{n}$ is an isometry. It is a celebrated result of Ruan that every abstract operator space is completely isometric with a {\it concrete operator space}, i.e. a closed subspace of $B(H)$ for a Hilbert space $H$  \cite[Theorem 2.3.5]{ER} . In the latter case, the matrix norms are given by the canonical identification $M_n(B({H}))\cong B(H^n)$.

Given two operator spaces $V$ and $W$, we let $CB(V,W)$ denote the space of all completely
bounded maps from $V$ to $W$. Then $CB(V,W)$ becomes a Banach space with respect to the norm
$\norm{ \cdot }_{cb}$ and is in fact an operator space via the identification $M_n(CB(V,W))\cong
CB(V,M_n(W))$. This, in particular, induces a canonical operator space structure on $V^*$ (\cite[Section 3.2]{ER}).

It is well-known that every Banach space can be given an operator
space structure, though not necessarily in a unique way. The smallest and largest matrix norms that can be considered on a Banach space $V$ are called the {\it minimal} and {\it maximal} operator space structure; they are denoted by Min $V$ and Max $V$. It can be shown that for any operator space $W$ and bounded maps $\theta: Max\ V \to W$ and $\vartheta:W \to Min\ V$, both $\theta$ and $\vartheta$ are completely bounded with $\norm{\theta}_{cb}=\norm{\theta}$ and $\norm{\vartheta}_{cb}=\norm{\vartheta}$ (\cite[Section 3.3]{ER}).

Given two operator spaces $V$ and $W$, there are many ways to define operator space matrix norm on the
algebraic tensor product $V\otimes W$. In this paper, we consider two of them: the {\it operator space projective tensor product} $V\widehat{\otimes}\,W$ and the {\it Haagerup tensor product} $V\otimes^h W$.
(see \cite[Sections 7 and 8]{ER} for definition and more details). We highlight the following canonical complete isometry for all operators spaces $V$, $W$, $Z$
\begin{align}\label{Eq:Op proj tensor product-isomorphism}
CB(V\widehat{\otimes}\ W, Z)\cong CB(V, CB(W, Z)).
\end{align}
We also note that if $E,F$ are Banach spaces and $E \otimes^\gamma F$ is their Banach space projective tensor product, then we have the complete isometric identification (\cite[Eq. (8.2.6)]{ER}):
\begin{align}\label{Eq:Op proj tensor product-Max space}
Max\ F \widehat{\otimes}\ Max\ E \cong {Max}\ (E\otimes^\gamma F).
\end{align}


A Banach algebra $A$ that is also an operator space is called a
\textit{quantized Banach algebra} if the multiplication map
\[m:A\widehat{\otimes}A\rightarrow A,\; u\otimes v \mapsto uv\]
is completely bounded. Moreover, an operator space $X$ is called a
\textit{completely bounded $A$-bimodule}, if $X$ is a Banach $A$-bimodule and if the maps
\[A\widehat{\otimes} X\rightarrow X \ \ , \ \ u\otimes x \mapsto ux \]
and
\[X\widehat{\otimes} A\rightarrow X \ \ , \ \ x\otimes u \mapsto xu \]
are completely bounded. In general, if $X$ is a completely bounded $A$-bimodule, then its dual
space $X^*$ is a completely bounded $A$-bimodule via the actions
\[(u\cdot x^*)(x)=x^*(xu) \ \ \ , \ \ \ (x^*\cdot u)(x)=x^*(ux)\]
for every $u\in A$, $x\in X$, and $x^*\in X^*$. We note that for any Banach algebra $A$, Max\ $A$ is always a quantized Banach algebras and any Banach $A$-bimodule is a complectly bounded Max\ $A$-bimodule.

For a Hilbert space $H$, we let $H_c$ and $H_r$ denote the {\it column} and {\it row Hilbert operator spaces} on $H$, respectively (see \cite[Section 3.4]{ER}). If we let $\overline{H}$ to be the conjugate Hilbert space of $H$, then we have the following complete isometries \cite[Theorem 3.4.1 and Proposition 3.4.2]{ER}:
\begin{align}\label{Eq:column row-CB representations}
B(H)\cong CB(H_c) \ \ \text{and}\ \ B(\overline{H})\cong CB(H_r).
\end{align}
Here the first the mapping is the identity map whereas the second one is given by
\begin{align}\label{Eq:row-CB representations-formulation}
\la \overline{T}(\bar{h}) , \bar{k} \ra_{\overline{H}}=\la h, T(k) \ra_{H} \ \ (T\in CB(H_r)).
\end{align}
Finally, by \cite[Proposition 9.3.2]{ER}, for every operator space $V$, we have the following complete isometries:
\begin{align}\label{Eq:column row Hagg tensor product}
V\otimes^h H_c\cong V\widehat{\otimes} H_c \ \ , \ \ H_r\otimes^h V\cong H_r\widehat{\otimes} V.
\end{align}



\subsection{Twisted Orlicz algebras as operator algebras}

Let A be a quantized Banach algebra. We say that $A$ is {\it (completely) isomorphic to an operator algebra} if there is an operator algebra (i.e. closed subalgebra) $B\subseteq B(H)$ and a (completely) bounded linear isomorphism $\rho: A \to B$ such that $\rho^{-1}:B\to A$ is also (completely) bounded. There are interesting examples of (nontrivial) quantized Banach algebras isomorphic to operator algebra. For instance, for all $1\leq p \leq \infty$, the spaces Min $l^p$ $(1\leq p \leq \infty)$ with pointwise product and the Schatten Spaces Min $S_p$ endowed with the Schur product are operator algebras (\cite[Corollary 5.4.11]{BL}), \cite{BBLV}, \cite{Dav}, \cite{V}).

In \cite{V}, Varopoulos showed that certain weighted group algebras on integers are isomorphic to operator algebras. His results were generalized in \cite{LSS} where it was shown that $l^1(G,\om)$ is isomorphic to an operator algebra when $G$ is a finitely generated group of polynomial growth and $\om$ is either the polynomial weight \eqref{Eq:poly weight-defn} with $1/\om \in l^2(G)$ or the subexponential weight \eqref{Eq:Expo weight-defn} (see \cite[Theorem 3.1 and Theorem 3.5]{LSS}). We point out that since every operator algebra is Arens regular and isomorphism preserve Arens regularity, one cannot extend the results of \cite{LSS} to a non-discrete case \cite{CY}.

In this section, we show that a subclass of twisted Orlicz algebras which were shown to be Arens regular in \cite{OS2} are in fact completely isomorphic to operator algebras. Interestingly, in contrast to the weighted group algebras, it is not necessary for the underlying group to be discrete to obtain our result. We first need the following two technical lemma.

\begin{lem}\label{L:column row multi-L1 space}
Let $G$ be a locally compact group, and let $u\in L^2(G)$. Then the mappings $R_c:L_c^2(G)\to L^1(G)$ and $R_r:L_r^2(G)\to L^1(G)$ defined by
\begin{align}\label{Eq:cb mult-row or Column space}
 R_c(f)=R_r(f)=fu \ \ \ \ \ (f\in L^2(G))
 \end{align}
are completely bounded. Here $L_c^2(G)$ and $L_r^2(G)$ are column and row Hilbert operator space on $L^2(G)$, respectively, and the operator space structure on $L^1(G)$ is the maximal operator space.
\end{lem}

\begin{proof}
We know that the mappings
$$L: L^\infty(G)\to B(L^2(G)) \ , \ L(\xi)(f)=f\xi $$
and $$\overline{L}: L^\infty(G)\to B(\overline{L^2(G)}) \ , \ L(\xi)(\bar{f})=\bar{f}\xi $$
are $*$-isomorphisms between von Neumann algebras so that they are complete isometries. By applying the identifications \eqref{Eq:Op proj tensor product-isomorphism}, \eqref{Eq:column row-CB representations}, and \eqref{Eq:row-CB representations-formulation}, it follows routinely that the mappings
$$m_c: L^\infty(G)\widehat{\otimes} L^2_c(G) \to L_c^2(G) \ , \ \xi\otimes f\mapsto f\xi $$
and $$m_r: L^\infty(G)\widehat{\otimes} L^2_r(G) \to L_r^2(G) \ , \ \xi\otimes f\mapsto f\xi$$
are completely bounded. In other worlds, $L^2_c(G)$ and $L^2_r(G)$ are operator $L^\infty(G)$-bimodule and so are their dual $L^2_c(G)^*$ and $L^2_r(G)^*$. In particular, for every $F\in L^2(G)^*$, the mappings
$$L^\infty(G)\to L_c^2(G)^* \ , \ \xi \mapsto \xi \cdot F $$
and $$L^\infty(G) \to L_r^2(G)^* \ , \ \xi \mapsto \xi \cdot F$$
are completely bounded, where $(\xi\cdot F)(f)=F(\xi f)$ ($f\in L^2(G)$).  However, a straightforward computation together with Riesz representation theorem shows that the preceding maps are exactly $R_c^*$ and $R_r^*$, respectively, where $R_c$ and $R_r$ are the mappings defined in \eqref{Eq:cb mult-row or Column space} and $F\in L^2(G)^*$ is defined by $F(f)=\int_G f(s)u(s)ds$ ($f\in L^2(G)$). Thus, by \cite[Proposition 3.2.2]{ER}, $R_c$ and $R_r$ are completely bounded.
\end{proof}

\begin{lem}\label{L:covolution style maps-Orlicz spaces}
Let $G$ be a locally compact group, and $L\in L^\infty(G\times G)$. Then the operators
\begin{align}\label{Eq:covolution style maps-Orlicz spaces-I}
m_{\Phi,1}:L^\Phi(G)\widehat{\otimes} L^1(G) \to L^\Phi(G)
\end{align}
and
\begin{align}\label{Eq:covolution style maps-Orlicz spaces-II}
m_{1,\Phi}:L^1(G)\widehat{\otimes} L^\Phi(G) \to L^\Phi(G)
\end{align}
given by $(f\in L^\Phi(G)$ and $g\in L^1(G))$
\begin{align*}
m_{\Phi,1}(f\otimes g)(t)=\int_G f(s)g(s^{-1}t)L(s, s^{-1}t) dt \ \ (s\in G),
\end{align*}
and
\begin{align*}
m_{1,\Phi}(g\otimes f)(t)=\int_G f(s)g(s^{-1}t)L(s, s^{-1}t) dt \ \ (s\in G),
\end{align*}
are well-defined and completely bounded. Here the operator space structures on $L^\Phi(G)$ and $L^1(G)$ are  the maximal operator space.
\end{lem}

\begin{proof}
We will show that $m_{\Phi,1}$ is well-defined and completely bounded. The other case can be proven similarly.

Without loss of generality, we can assume that $\|L\|_\infty\leq 1$. Since $L^\Phi(G)$ and $L^1(G)$ are maximal operator spaces, by \eqref{Eq:Op proj tensor product-Max space}, it suffices to show that $m_{\Phi,1}$ is defined and bounded on the Banach space projective tensor product $L^\Phi(G){\otimes^\gamma} L^1(G)$. To this end, for every $f\in L^\Phi(G)$ and $g\in L^1(G)$, a routine calculation shows that
$$|m_{\Phi,1}(f\otimes g)|\leq |f|*|g|.$$
Therefore, by \eqref{Eq:Orlicz Luxemburg defn}, Theorem \ref{T:twisted Orlicz alg}(i) and the fact that $\Phi$ is increasing, we have that
$$N_\Phi(m_{\Phi,1}(f\otimes g))|\leq N_\Phi(|f|*|g|)\leq C N_\Phi(f)\|g\|_1,$$
where $C>0$ is a constant independent of $f$ and $g$. This complete the proof.
\end{proof}


The following theorem is the main result of this section whose proof relies on two crucial facts: a decomposition of the twisted convolution on Orlicz spaces similar to the one presented in
\cite[Theorem 5.2]{OS2} as well as the modern approaches of characterization of abstract operator algebras using the theory of operator spaces. In the latter case, we use the well-known result of D. Blecher that a quantized Banach algebra $A$ is completely isomorphic to an operator algebra if and only if the multiplication mapping from $A\otimes A$ into $A$ extends to a completely bounded map from $A\otimes^h A$ into $A$ (\cite[Theorem 5.2.1]{BL}).


\begin{thm}\label{T:twisted Orlicz alg-operator alg}
Let $G$ be a locally compact group, and let $\Om \in \Zb$. Suppose that $L^\Phi(G)\subseteq L^2(G)$ and there exit non-negative measurable functions $u,v \in L^2(G)$ satisfying \eqref{Eq:2-cocycle bdd sum}. Then $(L^\Phi(G),\tw)$, with the maximal operator space structure, is a twisted Orlicz algebra which is completely isomorphic to an operator algebra.
\end{thm}

\begin{proof}
We first note that since $L^\Phi(G)\subseteq L^2(G)$, a standard application of the closed graph theorem shows that the inclusion is bounded so that there is $C>0$ such that
$$\|f\|_2\leq C N_\Phi(f) \ \ \ (f\in L^\Phi(G)).$$
Moreover, by \cite[Lemma 6.2.1 and Corollary 6.2.2]{rao}, $L^2(G)\subseteq \Sm^\Psi(G)$.
In particular, by our hypothesis and Theorem \ref{T:twisted Orlicz alg}, $(L^\Phi(G),\tw)$ is a twisted Orlicz algebra. Now suppose that
$$\Gamma: L^\Phi(G) \widehat{\otimes} L^\Phi(G) \to L^\Phi(G) \ \ , \ \ \Gamma(f\otimes g)=f\tw g,$$
is the multiplication map on $L^\Phi(G)$. By our hypothesis and \eqref{Eq:2-cocycle bdd sum}, we can take an element $L\in L^\infty(G\times G)$ with $\|L\|_\infty\leq 1$ such that
$$\Om(s,t)=L(s,t)\Big(u(s)+v(t)\Big) \ \ \ (s,t\in G).$$
Hence for every $f,g\in L^\Phi(G)$, we have
\begin{eqnarray*}
(f \tw g)(t)&=& \int_G f(s)g(s^{-1}t)\Om(s,s^{-1}t) ds \\
&=& \int_G f(s)u(s) g(s^{-1}t)L(s,t) ds+\int_G f(s)g(s^{-1}t)v(s^{-1}t)L(s,t) ds.
\end{eqnarray*}
Thus we can write
\begin{eqnarray*}
\Gamma=\Gamma_1+\Gamma_2,
\end{eqnarray*}
where
\begin{eqnarray*}
\Gamma_1(f\tw g)(s):=\int_G f(s)u(s) g(s^{-1}t)L(s,s^{-1}t) ds
\end{eqnarray*}
and
\begin{eqnarray*}
\Gamma_2(f\tw g)(s):=\int_G f(s)g(s^{-1}t)v(s^{-1}t)L(s,s^{-1}t) ds.
\end{eqnarray*}
Therefore, to show that $L^\Phi(G)$ is completely isomorphic to an operator algebra, by Blecher's result (\cite[Theorem 5.2.1]{BL}), it suffices to show that $\Gamma_i$, $i=1,2$ extend to completely bounded maps from $L^\Phi(G)\otimes^h L^\Phi(G)$ into $L^\Phi(G)$. However this holds since we have the following chain of completely bounded maps that extend $\Gamma_1$ and $\Gamma_2$, respectively:
\begin{eqnarray*}
\Gamma_1: L^\Phi(G)\otimes^h L^\Phi(G) &\stackrel{\iota_{\Phi,r}\otimes id}{\longrightarrow}& L^2_r(G)\otimes^h L^\Phi(G)\\ &\cong& L^2_r(G)\widehat{\otimes} L^\Phi(G) \stackrel{R_c\otimes id}{\longrightarrow} L^1(G) \widehat{\otimes} L^\Phi(G) \stackrel{m_{1,\Phi}}{\longrightarrow}  L^\Phi(G)
\end{eqnarray*}
and
\begin{eqnarray*}
\Gamma_2: L^\Phi(G)\otimes^h L^\Phi(G) &\stackrel{id\otimes \iota_{\Phi,c}}{\longrightarrow}& L^\Phi(G)\otimes^h L^2_c(G) \\ &\cong& L^\Phi(G)\widehat{\otimes} L^2_c(G) \stackrel{id\otimes R_c}{\longrightarrow} L^\Phi(G)\widehat{\otimes} L^1(G) \stackrel{m_{\Phi,1}}{\longrightarrow}  L^\Phi(G).
\end{eqnarray*}
Here the mappings $R_c$, $R_r$, $m_{\Phi,1}$, and $m_{1,\Phi}$ are defined in \eqref{Eq:cb mult-row or Column space}, \eqref{Eq:covolution style maps-Orlicz spaces-I}, and \eqref{Eq:covolution style maps-Orlicz spaces-II} which are shown to be completely bounded in Lemma \ref{L:column row multi-L1 space} and Lemma \ref{L:covolution style maps-Orlicz spaces}. Also, we are using \eqref{Eq:column row Hagg tensor product} and the facts that formal identities $\iota_{\Phi,r}: L^\Phi(G)\to L_r^2(G)$ and $\iota_{\Phi,c}: L^\Phi(G)\to L_c^2(G)$ are completely bounded since they are bounded and $L^\Phi(G)$ and $L^1(G)$ are maximal operator spaces.
\end{proof}



In the rest of this section, we will show how one can obtain a wide range of twisted Orlicz algebras satisfying the assumption of Theorem \ref{T:twisted Orlicz alg-operator alg} so that they are isomorphic to operator algebras. We start by looking at how one could embed an Orlicz space into the space of $L^2$-integrable functions, a necessary condition in Theorem \ref{T:twisted Orlicz alg-operator alg}. Luckily, using the results of \cite{rao}, we can formulate a straightforward criterion to verify the existence of such embedding.

\begin{prop}\label{P:Orlicz space subspace L2}
Let $G$ be a locally compact group. Then $L^\Phi(G)\subseteq L^2(G)$ if either one of the following conditions holds:\\
$(i)$ $G$ is compact and there is $K>0$ and $x_0\geq 0$ such that
\begin{equation}\label{Eq:Orlicz space subspace L2-compact}
Kx^2 \leq \Phi(x) \ \ \text{for all}\ \ (x\geq x_0 \geq 0);
\end{equation}
$(ii)$ $G$ is discrete and there is $K>0$ and $x_0> 0$ such that
\begin{align}\label{Eq:Orlicz space subspace L2-discrete}
Kx^2 \leq \Phi(x) \ \ \text{for all}\ \ (x_0\geq x \geq 0);
\end{align}
$(iii)$ $G$ is noncompact and there is $K>0$ such that
\begin{align}\label{Eq:Orlicz space subspace L2-non compact}
Kx^2 \leq \Phi(x) \ \ \text{for all}\ \ (x \geq 0).
\end{align}
\end{prop}

\begin{proof}
Parts (i) and (iii) follow from \cite[Theorem 5.1.3]{rao}. To prove (ii),
let $f\in l^\Phi(G)$. By \eqref{Eq:Orlicz defn},
there is $\alpha>0$ such that
$\sum_{s\in G} \Phi(\alpha |f(s)|)<\infty$ implying that $\lim_{s\to \infty} f(s)=0$ as $\Phi$ is continuous and strictly increasing. Thus, by \eqref{Eq:Orlicz space subspace L2-discrete}, there is $K>0$ such that
$$K\alpha^2|f(s)|^2\leq \Phi(\alpha |f(s)|)\ \ \ \ (s\in G).$$
Therefore,
$$K\alpha^2\sum_{s\in G} |f(s)|^2 \leq \sum_{s\in G} \Phi(\alpha |f(s)|)<\infty,$$
i.e. $f\in l^2(G)$.
\end{proof}

\begin{rem}\label{R:Young functions equivalent with parabola}
We can construct a large family of Young functions satisfying \eqref{Eq:Orlicz space subspace L2-compact} or \eqref{Eq:Orlicz space subspace L2-discrete}. Let $(\Phi,\Psi)$ be a complementary pair of continuous strictly increasing Young functions. Define $\Phi_0(x)=\Phi(x^2)$ $(x\geq 0)$. It is clear that $\Phi_0$
is a continuous strictly increasing Young function whose complementary Young function $\Psi_0$ also has the same property (see the discussion in
\cite[Page 10]{rao}). Moreover,
$$\Phi(1)x^2\leq \Phi(x^2)=\Phi_0(x) \ \ \ (x\geq 1).$$
Thus, $\Phi_0$ satisfies \eqref{Eq:Orlicz space subspace L2-compact}. On the other hand, we can show that $\Psi_0$ satisfies~\eqref{Eq:Orlicz space subspace L2-discrete}. For this we choose $a>0$ such that $\Phi(1)<\frac1a$, and then for $0\le x\le\frac1a$ we will have
$$
\Phi_0(ax)=\Phi(a^2x^2)\le (\Phi(1)\,a^2)x^2,\quad \text{where}\ \ \Phi(1)\,a^2<a.
$$
Combining this with the Young inequality \eqref{Eq:Young inequality}, we obtain
$$
ax^2 \leq \Phi_0(ax)+\Psi_0(x)\leq (\Phi(1)\,a^2)x^2+\Psi_0(x)\quad \left(0\le x\le\frac1a\right).
$$
Hence,
$$
(a-\Phi(1)\,a^2)x^2 \leq \Psi_0(x) \quad \left(0\leq x\leq \frac1a\right).
$$
Similar results can also be obtained if we replace $\Phi_0$ with the Young function $\Phi_1(x)=\Phi(x)^2$ and $\Psi_0$ with $\Psi_1$, the complementary Young function of $\Phi_1$.
\end{rem}

We can now present many classes of twisted Orlicz algebras completely isomorphic to operator algebras on compactly generated groups of polynomial growth. In the case where the group is discrete, this can be compared with \cite[Theorem 3.1 and Theorem 3.5]{LSS}.

\begin{cor}\label{C:twisted lp alg-poly and exp weight-Operator alg}
Let $G$ be a compactly generated group of polynomial growth, $\Om\in \Zbw$,
and $\om$ be the weight associated to $\Om$. Suppose that $\om$ is either of the following weights:\\
$(i)$ $\om=\om_\beta$, the polynomial weight \eqref{Eq:poly weight-defn} with $1/\om \in L^2(G)$;\\
$(ii)$ $\om=\sg_{\alpha,C}$, the subexponential weight \eqref{Eq:Expo weight-defn};\\
$(iii)$ $\om=\rho_{\gamma,C}$, the subexponential weight \eqref{Eq:Expo weight II-defn}.

Then the following holds:\\
$(1)$ $(L^2(G)\cap L^\Phi(G),\tw)$ is completely isomorphic to an operator algebra.\\
$(2)$ If $G$ is compact and $\Phi$ satisfies \eqref{Eq:Orlicz space subspace L2-compact}, then $(L^\Phi(G),\tw)$ is completely isomorphic to an operator algebra.\\
$(3)$ If $G$ is discrete and $\Phi$ satisfies \eqref{Eq:Orlicz space subspace L2-discrete}, then $(l^\Phi(G),\tw)$ is completely isomorphic to an operator algebra.
\end{cor}

\begin{proof}
For (1), let $\tilde{\Phi}(x)=x^2+\Phi(x)$ $(x\geq 0)$. It is clear that $\tilde{\Phi}$ is a strictly increasing continuous Young function whose complementary Young function is also continuous and strictly increasing
\cite[Corollary 1.3.2 and Theorem 1.3.3]{rao}. Also, $L^2(G)\cap L^\Phi(G)$ is the Orlicz space associated to $\tilde{\Phi}$. Therefore, the result follows from Proposition \ref{P:Orlicz space subspace L2}(iii) and Theorem \ref{T:twisted Orlicz alg-operator alg}. The statements (2) and (3) follow from part (1) and
Proposition \ref{P:Orlicz space subspace L2}(i) and (ii).
\end{proof}

The method presented in Remark \ref{R:Young functions equivalent with parabola} is very useful in constructing Young functions satisfying \eqref{Eq:Orlicz space subspace L2-compact} or \eqref{Eq:Orlicz space subspace L2-discrete}. However, this is certainly not the only way as we see below.

\begin{rem}\label{Young functions-second derivative-operator alg}
Suppose that $\Phi''$ exists on $\R^+$. Since $\Phi'$ is continuous on $\R^+$ and $\Psi$ is strictly increasing and continuous, it follows from \cite[Corollary 1.3.2 and Theorem 1.3.3]{rao} that $\Phi_+'(0)=0$ and $\lim_{x\to \infty} \Phi'(x)=\infty$. Thus, by applying repeatedly the L'Hospital's rule, we get the following:

$(i)$ $\Phi$ satisfies \eqref{Eq:Orlicz space subspace L2-compact} if $\lim_{x\to \infty} \Phi''(x)\neq 0$;

$(ii)$ $\Phi$ satisfies \eqref{Eq:Orlicz space subspace L2-discrete} if
$\Phi_+''(0)\neq 0$;

$(iii)$ $\Phi$ satisfies \eqref{Eq:Orlicz space subspace L2-non compact} iff
 $\Phi$ satisfies both \eqref{Eq:Orlicz space subspace L2-compact} and  \eqref{Eq:Orlicz space subspace L2-discrete}.

 We can apply the above criterions to various Young functions such as a few we point out below:

(1) If $\Phi(x)=x^2/2+x^p/p$, then $L^\Phi(G)=L^2(G)\cap L^p(G)$. In this case, $\Phi$ satisfies  \eqref{Eq:Orlicz space subspace L2-compact} if $p\geq 2$ and it satisfies \eqref{Eq:Orlicz space subspace L2-discrete} if $1< p \leq 2$.

(2) If $\Phi(x)=x^\alpha\ln (1+x)$ $(\alpha \geq 1$), then $\Phi$ satisfies  \eqref{Eq:Orlicz space subspace L2-compact} if $\alpha >2$ and  it satisfies \eqref{Eq:Orlicz space subspace L2-discrete} if $1\leq \alpha \leq 2$.

(3) If $\Phi(x)=\cosh x-1$, then $\Phi$ satisfies \eqref{Eq:Orlicz space subspace L2-non compact}.

(4) If $\Phi(x)=e^x-x-1$, then $\Phi$ satisfies \eqref{Eq:Orlicz space subspace L2-non compact}.

(5) If $\Phi(x)=(1+x)\ln(1+x)-x$, then $\Phi$ satisfies \eqref{Eq:Orlicz space subspace L2-discrete}.

\end{rem}








\begin{exm}\label{E:twisted lp alg on integers-poly and exp weight-Operator alg}
Let $\Z^d$ be the group of $d$-dimensional integers.
A usual choice of generating set for $\Z^d$ is
	$$F=\{(x_1,\ldots, x_d) \mid x_i\in \{-1,0,1\} \}.$$
It is straightforward to see that
\begin{align}\label{Eq:growth generating set of integers}
|F^n|=(2n+1)^d \ \ \ (n=0,1,2,\ldots).
\end{align}
Now suppose that $1< p,q<\infty$ with $\frac{1}{p}+\frac{1}{q}=1$. Let $\om_\beta$ be the polynomial weight on $\Z^d$ defined in \eqref{Eq:poly weight-defn}. Then, by Theorem~\ref{T:subexpo weight decreasing quotient} and \cite[Theorem 3.2]{K3}, $l^p_{\om_\beta}(\Z^d)$ is a Banach algebra if and only if $\beta >d/q$. Thus, by Corollary \ref{C:twisted lp alg-poly and exp weight-Operator alg}, $l^p_{\om_\beta}(\Z^d)$ is completely isomorphic to an operator algebra if $1<p\leq 2$ and $\beta >d/2.$
On the other hand, $l^p_\om(\Z^d)$ is always completely isomorphic to an operator algebra if
$\om$ is either of the subexponential weights \eqref{Eq:Expo weight-defn} or \eqref{Eq:Expo weight II-defn}.
\end{exm}

\appendix

\section{An example}

\begin{exm}
We build a function $\rho:[0,\infty)\to[0,\infty)$ satisfying the following:\\
(i) $\rho(0)=0$ and $\rho$ is increasing;\\
(ii) $\rho$ is concave;\\
(iii) $\displaystyle \lim_{x\to\infty}\,\frac{\rho(x)}{x}=0$;\\
(iv) $\sum\limits_{n=1}^{\infty}\, e^{-\rho(n)}$ converges;\\
(v) $\sum\limits_{n=1}^{\infty}\, e^{\rho(2n)-2\rho(n)}$ diverges.
\end{exm}

\begin{proof}
In order for (iv) to be satisfied, we make sure that $\rho$ satisfies
\begin{equation}\label{1}
\rho(n)\ge2\ln n, \quad n\in\mathbb{N}.
\end{equation}
In this case $e^{-\rho(n)}\le e^{-2\ln n}=n^{-2}$, and so the series from (iv) will be majorized by convergent series $\sum\limits_{n=1}^{\infty} \frac{1}{n^2}$.

Next we notice that concavity of $\rho$ implies that the sequence $e^{\rho(2n)-2\rho(n)}$ is non-increasing. To prove this, it would be enough to show that for any $n\in\mathbb{N}\cup\{0\}$
$$
2\rho(n+1)-\rho(2n+2)\ge 2\rho(n)-\rho(2n)\quad\Leftrightarrow
$$
$$
2(\rho(n+1)-\rho(n))\ge(\rho(2n+2)-\rho(2n+1))+(\rho(2n+1)-\rho(2n)).
$$
Since $\rho$ is concave, we have that $\{\rho(n+1)-\rho(n)\}$ is a non-increasing sequence, which readily implies the above inequality.

It is known that for a non-increasing sequence $a_n\ge0$ convergence of the series $\sum\limits_{n=1}^{\infty} a_n$ implies that $\lim\limits_{n\to\infty} na_n=0$. Hence, in order to satisfy (v), it would be enough to make sure that $\lim\limits_{n\to\infty} ne^{\rho(2n)-2\rho(n)}\ne0$. For this, in turn, it would be enough to have the following:
\begin{equation}\label{2}
\forall N>0\  \exists\  n\ge N: 2\rho(n)-\rho(2n)=\ln n,
\end{equation}
in which case $ne^{\rho(2n)-2\rho(n)}=ne^{-\ln n}=1$.

We remark the following geometrical interpretation of the quantity $2\rho(n)-\rho(2n)$: if we connect the points $(n,\rho(n))$ and $(2n,\rho(2n))$ on the graph of $\rho$ with a straight line, then its $y$-intercept ($y$-coordinate of the point of intersection with the $y$-axis) will be equal to $2\rho(n)-\rho(2n)$. With this in mind, we pick any $n_1>2$ and consider the point $P_1(0,\ln n_1)$ on the $y$-axis. We build a tangent line to the graph of $2\ln x$ through $P_1$. Let this tangent touch the curve $y=2\ln x$ at the point $(x_1,2\ln x_1)$. Then, on one hand, the slope of this tangent is $2/x_1$, and on the other hand, it is equal to $(2\ln x_1-\ln n_1)/(x_1-0)$:
$$
\frac{2\ln x_1-\ln n_1}{x_1}=\frac{2}{x_1}\ \ \Rightarrow\ \ 2\ln x_1=\ln n_1+2\ \ \Rightarrow\ \ x_1=e^{\frac12\ln n_1+1}=e\sqrt n_1.
$$
We then define $\rho(x)$ on $[n_1,2n_1]$ so that its graph coincides with our tangent:
\begin{equation}\label{3}
\rho(x)=\ln n_1+\frac{2x}{e\sqrt{n_1}},\ x\in[n_1,2n_1].
\end{equation}
It follows that
\begin{equation}\label{5}
2\rho(n_1)-\rho(2n_1)=\ln n_1,
\end{equation}
and we also have
\begin{equation}\label{4}
\frac{\rho(n_1)}{n_1}=\frac{\ln n_1+\frac{2\sqrt{n_1}}{e}}{n_1}.
\end{equation}
We now want to choose $n_2>2n_1$ and build $\rho$ on $[n_2,2n_2]$ in the same way as we built it on $[n_1,2n_1]$. We need to make sure that the pieces of tangents can be connected in such a way that $\rho$ is concave on $[n_1,2n_2]$. The slope of $\rho$ on $[n_1,2n_1]$ is $\frac{2}{e\sqrt{n_1}}$ and the slope of $\rho$ on $[n_2,2n_2]$ is $\frac{2}{e\sqrt{n_2}}<\frac{2}{e\sqrt{n_1}}$. This means that we only need to make sure that the line containing the graph of $\rho$ on $[n_1,2n_1]$ and the line containing the graph of $\rho$ on $[n_2,2n_2]$ intersect at a point whose $x$-coordinate is between $2n_1$ and $n_2$. From (\ref{3}) we get that the coordinates $(t_1,s_1)$ of this point of intersection satisfy the following:
$$
s_1=\ln n_1+\frac{2t_1}{e\sqrt{n_1}}=\ln n_2+\frac{2t_1}{e\sqrt{n_2}}.
$$
Hence,
$$
t_1=\frac{\ln n_2-\ln n_1}{\frac{2}{e\sqrt{n_1}}-\frac{2}{e\sqrt{n_2}}}=\frac{e\ln\frac{n_2}{n_1}\sqrt{n_1n_2}}{2(\sqrt{n_2}-\sqrt{n_1})} =\frac{e\ln\frac{n_2}{n_1}(n_1\sqrt{n_2}+n_2\sqrt{n_1})}{2(n_2-n_1)}.
$$
We want $2n_1<t_1<n_2$, i.e.,
$$
4n_2n_1-4n_1^2\le e\ln\frac{n_2}{n_1}(n_1\sqrt{n_2}+n_2\sqrt{n_1})\le 2n_2^2-2n_2n_1.
$$
Comparing the orders of growth in $n_2$ in the three formulas above, we see that this inequality will be satisfied for large enough $n_2$. This guarantees that if we define $\rho$ on $(2n_1,n_2)$ by
$$
\rho(x)=\left\{\begin{array}{ll}
\ln n_1+\frac{2x}{e\sqrt{n_1}},\ & x\in(2n_1,t_1]\\
\ln n_2+\frac{2x}{e\sqrt{n_2}},\ & x\in(t_1,n_2)
\end{array}\right.,
$$
then $\rho$ will be concave on $[n_1,2n_2]$.

We can continue our process in the same way by choosing $n_3,n_4,\ldots$ and defining piecewise linear $\rho$ in the same manner. We can also define $\rho$ on $[0,n_1)$ so that it satisfies (i) and (ii).

Now we check that conditions (i)-(v) are satisfied for our function $\rho$. We already verified (ii), and (i) is fulfilled since the slopes of $\rho$ on all $[n_k,2n_k]$ are positive. Because (\ref{4}) holds for all $n_k$ and $n_k\to \infty$ ($n_{k+1}\ge2n_k$), we have that $\frac{\rho(n_k)}{n_k}\to0$. Since we already proved that $\rho$ is concave and $\rho(0)=0$, we know that $\frac{\rho(n)}{n}$ is decreasing, and so (iii) follows. Our function $\rho$ satisfies (\ref{1}) because the function $2\ln x$ is concave, and hence lies below all of its tangents. This asserts that (iv) holds for $\rho$. Finally, since (\ref{5}) takes place for every $n_k$ and $n_k\to\infty$, we have (\ref{2}) which implies (v).
\end{proof}

\end{document}